\documentclass[11pt]{amsart}
\usepackage[centertags]{amsmath}
\usepackage{amsfonts,amssymb}
\usepackage{graphicx}
\usepackage{enumerate}

  \newtheorem{theorem}{Theorem}
  \newtheorem{corollary}{Corollary}
  
  \newtheorem{lemma}{Lemma}

\begin{document}

\title{Integers with large practical component}
\author{Andreas Weingartner}
\address{ 
Department of Mathematics,
351 West University Boulevard,
 Southern Utah University,
Cedar City, Utah 84720, USA}
\email{weingartner@suu.edu}

\subjclass[2010]{11N25, 11N37}

\begin{abstract}
A positive integer $n$ is called practical if all integers between $1$ and $n$ can be written as a sum of distinct divisors of $n$.
We give an asymptotic estimate for the number of integers $\le x$ which have a practical divisor $\ge y$.
\end{abstract} 
\maketitle

\section{Introduction}

A positive integer $n$ is called \emph{practical} if all integers between $1$ and $n$ can be written as a sum of distinct divisors of $n$.
In 1948, Srinivasan \cite{Sri} began the study of practical numbers, which have been the source of a fair amount of research activity ever since.  
Let $P(x)$ denote the number of practical numbers $\le x$. 
Increasingly precise estimates for $P(x)$ have been obtained by Erd\H{o}s and Loxton \cite{EL},
Hausman and Shapiro \cite{HS}, Margenstern \cite{Mar}, Tenenbaum \cite{Ten86} and Saias \cite{Sai}, 
who found that the order of magnitude of $P(x)$ is $x/\log x$. 
In \cite{PDD} we showed that there is a positive constant $c$ such that
\begin{equation}\label{P}
P(x)=\frac{cx}{\log x} \left(1+O\left(\frac{\log \log x}{\log x}\right) \right),
\end{equation}
confirming a conjecture by Margenstern \cite{Mar}.  In this note we want to generalize \eqref{P} to integers
which have a large practical divisor.

Let $g(n)$ denote the \emph{practical component} of $n$, i.e. the largest divisor of $n$ which is practical. 
We have $g(n)=n$ if and only if $n$ is practical, hence we can think of $g(n)$ as a measure for 
how close $n$ is to being practical. 
Let $M(x,y)$ be the number of integers $\le x$ whose practical component is at least $y$, i.e.
$$ M(x,y):=\#\{n\le x: g(n) \ge y \}.$$

A closely related arithmetic function is $f(n)$, the largest integer with the property that all integers in the interval 
$[1,  f(n)]$ can be written as a sum of distinct divisors of $n$. Clearly, $n$ is practical if and only if $f(n)\ge n$. 
Thus $f(n)$ represents another measure for how close $n$ is to being practical. Pollack and Thompson \cite{PT} call an integer $n$  
a \emph{practical pretender} (or a \emph{near-practical number}) if $f(n)$ is large. More precisely, they define
$$N(x,y):=\#\{n\le x: f(n) \ge y \}$$
and show that there are two positive constants $c_1, c_2$ such that
$$  c_1 \frac{x}{\log y} \le N(x,y) \le c_2 \frac{x}{\log y} \quad (4\le y \le x).$$ 
In \cite[Lemma 2.1]{PT}, they find that $f(n)$ satisfies $f(n)=\sigma(g(n))$, where $\sigma(m)$ denotes the sum
of the positive divisors of $m$.

To describe the asymptotic behavior of $M(x,y)$ and $N(x,y)$, we need the following notation.
Let $c$ be the positive constant in \eqref{P}, $\chi(n)$ be the characteristic function of the set of practical numbers,
$$ u = \frac{\log x}{\log y},$$
 and
$\omega(u)$ be Buchstab's function, i.e. the 
unique continuous solution to the equation
\begin{equation*}
(u\omega(u))' = \omega(u-1) \qquad (u>2)
\end{equation*}
with initial condition
$\omega(u)=1/u$ for $1\le u \le 2$.
\begin{theorem}\label{thmN}
 For $x\ge y \ge 2$ we have
\begin{enumerate}
\item[(i)]
$\displaystyle
M(x,y) = \frac{c(x \omega(u)-y)}{\log y}
+O\left(\frac{x\log \log 2y}{(\log y)^2}\right),
$
\item[]
\item[(ii)]
$\displaystyle
N(x,y) = \frac{cx \omega(u)}{\log y}
+O\left(\frac{y}{\log y} +\frac{x\log \log 2y}{(\log y)^2}\right),
$
\item[]
\item[(iii)] $\displaystyle M(x,y)= x \mu_y + O(2^y),$
\item[]
\item[(iv)] $\displaystyle N(x,y) = x \nu_y +O(2^y),$
\end{enumerate}
where 
$$  \mu_y:= 1-\sum_{n<y} \frac{\chi(n)}{n} \prod_{p\le \sigma(n)+1} \left(1-\frac{1}{p}\right) $$
and 
$$ \nu_y:= 1-\sum_{\sigma(n)<y} \frac{\chi(n)}{n} \prod_{p\le \sigma(n)+1} \left(1-\frac{1}{p}\right) . $$
\end{theorem}

\medskip

It may seem a little surprising to see Buchstab's function appear in the asymptotic formulas for $M(x,y)$ and $N(x,y)$.
The reason for this is that $M(x,y)$ and $N(x,y)$ satisfy functional equations (see Lemma \ref{lem1} below) which closely resemble 
the functional equation 
\begin{equation}\label{PFE}
\Phi(x,y) =  1 + \sum_{y<p\le x} \Phi(x/p, p-0)
\end{equation}
satisfied by
$$ \Phi(x,y) := \# \{ n\le x : P^-(n)>y \}  .$$
Here $P^-(n)$ denotes the smallest prime factor of $n$ and $P^-(1)=\infty$. The main difference is that the primes in \eqref{PFE}
are replaced by the practical numbers in Lemma \ref{lem1}, 
which explains the constant factor $c$ in Theorem \ref{thmN}. With Lemma \ref{Phi} (ii) we find 
that $M(x,y) \sim c \, \Phi(x,y)$ for $y\le (1-\varepsilon)x$ and $y\to \infty$.

Moreover, combining \eqref{P}, Theorem \ref{thmN}, Lemma \ref{Phi} and the prime number theorem, we have
$$ \frac{P(x)}{M(x,y)}\sim \frac{\pi(x)}{\Phi(x,y)} \sim \frac{1}{u\omega(u)} \qquad (y\to \infty,\ x/y \to \infty). $$
Hence the probability that a random integer $n\le x$ is practical, given that $g(n)\ge y$, is asymptotically equivalent to
the probability that a random integer $n\le x$ is prime, given that $P^-(n)>y$, as $y\to \infty$, $x/y \to \infty$.

The rapid convergence of $\omega(u)$ to $e^{-\gamma}$ (see Lemma \ref{omega} (ii)) and Theorem \ref{thmN} imply that, for $x\ge y\ge 2$,
\begin{equation}\label{MN}
M(x,y) ,N(x,y) = \frac{c e^{-\gamma} x}{\log y}
\left(1+O\left(\frac{1}{\Gamma(u+1)}+\frac{\log \log 2y}{\log y}\right)\right),
\end{equation}
where $\Gamma$ denotes the usual gamma function. Combining \eqref{MN} with (iii) and (iv) gives the estimate
$$ \mu_y, \nu_y=\frac{c e^{-\gamma} }{\log y} \left(1+O\left(\frac{\log \log y}{\log y}\right)\right) .$$
The following table shows  $\mu_y=\lim\limits_{x\to \infty} M(x,y)/x$ and $\nu_y=\lim\limits_{x\to \infty} N(x,y)/x$ for small values of $y$:

\medskip
\begin{center}
  \begin{tabular}{ |c | c | } 
    \hline
    $y \in $ & $\mu_y$ \\ \hline 
    $[0,1]$ & $1$  \\ \hline
    $(1,2]$ &  $1/2$ \\ \hline
    $(2,4]$ &  $1/3$   \\ \hline
    $(4,6]$ &  $29/105$   \\ \hline
  \end{tabular}
  \qquad
 \begin{tabular}{ |c | c | } 
    \hline
    $y \in $ & $\nu_y$ \\ \hline 
    $[0,1]$ & $1$  \\ \hline
    $(1,3]$ &  $1/2$ \\ \hline
    $(3,7]$ &  $1/3$   \\ \hline
    $(7,12]$ &  $29/105$   \\ \hline
  \end{tabular}
\end{center}

\medskip

From part (iii) of Theorem \ref{thmN} we obtain the natural density of integers whose practical component is equal to $m$.
\begin{corollary}\label{corM} Let $m\ge 1$ and
$$ \alpha_m := \mu_m-\mu_{m^+}= \frac{\chi(m)}{m} \prod_{p\le \sigma(m)+1} \left(1-\frac{1}{p}\right). $$
For $x\ge 1$ we have
$ \#\bigl\{n\le x : g(n)=m\bigr\} = x \alpha_m + O(2^m).$
\end{corollary}
\medskip

Pollack and Thompson \cite[Corollary 1.2]{PT} found that the set of integers $n$ with $f(n)=m$ has a natural density $\rho_m$.
Part (iv) of Theorem \ref{thmN} implies

\begin{corollary}\label{corN} Let $m\ge 1$ and
$$ \rho_m:= \nu_m-\nu_{m^+}=\sum_{\sigma(n)=m}  \frac{\chi(n)}{n} \prod_{p\le \sigma(n)+1} \left(1-\frac{1}{p}\right)
=\sum_{\sigma(n)=m} \alpha_n . $$
For $x\ge 1$ we have
$  \#\bigl\{n\le x : f(n)=m\bigr\} = x \rho_m + O(2^m) .$
\end{corollary}
\medskip

The following table shows non-zero values of $\alpha_m$ and $\rho_m$ for small $m$. 
Note that $\alpha_m >0$ if and only if $m$ is practical, while $\rho_m > 0$ if and only if $m=\sigma(n)$ for some practical number $n$.

\medskip
\begin{center}
  \begin{tabular}{ |c | c | } 
   \hline
  $m  $ & $\alpha_m$ \\ \hline 
    $1$ & $1/2$  \\ \hline
    $2$ &  $1/6$ \\ \hline
    $4$ &  $2/35$   \\ \hline
    $6$ &  $32/1001$   \\ \hline
  \end{tabular}
  \qquad
  \begin{tabular}{ |c | c | } 
    \hline
    $m  $ & $\rho_m$ \\ \hline 
    $1$ & $1/2$  \\ \hline
    $3$ &  $1/6$ \\ \hline
    $7$ &  $2/35$   \\ \hline
    $12$ &  $32/1001$   \\ \hline
  \end{tabular}
\end{center}
\medskip

The equality of $\alpha_m$ and $\rho_{\sigma(m)}$ does not always hold. For example, since $\sigma(54)=\sigma(56)=120$ and both $54$ and $56$ are practical, we have $\rho_{120}=\alpha_{54}+\alpha_{56}$.  Moreover, Pollack and Thompson \cite[Theorem 1.3]{PT} show that the number of integers $m\le x$ for which $\rho_m>0$ is $\ll \frac{x}{(\log x)^A}$ for every fixed $A>0$. 
Thus the support of $\rho_m$ is a much thinner set than the support of $\alpha_m$, the set of practical numbers.

The reader may have noticed that practical integers $n<y$ are not counted in $M(x,y)$. 
This suggests that we may want to consider replacing the parameter $y$ by an increasing function of $n$, so that smaller values of $n$ are not ignored. To this end, we define
$$ M_\lambda(x):= \#\{n\le x: g(n) \ge n^\lambda \}, \quad  N_\lambda(x):= \#\{n\le x: f(n) \ge n^\lambda \}.$$
Nevertheless, the following result shows that, for $x^\lambda \to \infty$,  $x^{1-\lambda} \to \infty$,
$$M_\lambda(x)\sim M(x,x^\lambda) \sim N_\lambda(x)\sim N(x,x^\lambda) \sim \frac{cx\omega(1/\lambda)}{\log(x^\lambda)}.$$

\begin{corollary}\label{pro}
For $x\ge  y \ge 2$  we have
\begin{enumerate}
\item[(i)]
$ \displaystyle M_{1/u}(x) =\frac{c x \omega(u)}{\log y}+O\left(\frac{x\log \log 2y}{(\log y)^2}\right),$
\item[]
\item[(ii)]
$ \displaystyle N_{1/u}(x) =\frac{c x \omega(u)}{\log y}+O\left(\frac{y}{\log y} + \frac{x\log\log 2y}{(\log y)^2}\right).$
\end{enumerate}
\end{corollary}

\bigskip

\section{Proofs}

Stewart \cite{Stew} and Sierpinski \cite{Sier} independently discovered the following characterization of practical numbers. 
An integer $n \ge 2$ with prime factorization $n=p_1^{\alpha_1} \cdots p_k^{\alpha_k}$, $p_1<p_2<\ldots < p_k$,
is practical if and only if 
$$p_j\le 1+\sigma\bigg( \prod_{1\le i \le j-1} p_i^{\alpha_i} \bigg) \qquad (1\le j \le k).$$
It follows that the practical component of $n$ is the largest practical divisor of $n$ of the form
$ \displaystyle \prod_{1\le i \le j} p_i^{\alpha_i}$.
 If $j<k$, i.e. $n$ is not practical, then we have $ \displaystyle p_{j+1}>1+\sigma\biggl( \prod_{1\le i \le j} p_i^{\alpha_i} \biggr)$.

\begin{lemma}\label{lem1}
For $x\ge 1$, $y\ge 1$ we have 
\begin{enumerate}[(i)]
\item $\displaystyle [x]  = \sum_{n\le x} \chi(n) \Phi\bigl(x/n, \sigma(n)+1\bigr)$
\item $\displaystyle M(x,y)  = \sum_{y\le  n\le x} \chi(n) \Phi\bigl(x/n, \sigma(n)+1\bigr) $
\item $\displaystyle N(x,y)  = \sum_{\substack{n\le x \\ \sigma(n)\ge y}} \chi(n) \Phi\bigl(x/n, \sigma(n)+1\bigr) $
\item $\displaystyle M_\lambda(x) =  \sum_{n\le x} \chi(n) \Phi\left(\min\left(x/n,n^{1/\lambda -1}\right), \sigma(n)+1\right) $
\item $\displaystyle N_\lambda(x) =  \sum_{n\le x} \chi(n) \Phi\left(\min\left(x/n,\sigma(n)^{1/\lambda }/n\right), \sigma(n)+1\right) $
\end{enumerate}
\end{lemma}

\begin{proof}
Each of these equations is based on the same principle, which is to count the integers $m$ contributing to the left-hand side according to their practical component $n$. Part (i) is Lemma 2.3 of \cite{PDD}. We only take a closer look at (ii). Every integer $m$ counted in $M(x,y)$ factors uniquely as $m=nr$, where $n$ is the practical
component of $m$, $n\ge y$ and $P^-(r)>\sigma(n)+1$. Given a practical component $n$, the number of admissible values of $r$ is given by $\Phi\bigl(x/n, \sigma(n)+1\bigr) $.
\end{proof}

\begin{lemma}\label{Phi}
We have
\begin{enumerate}
\item[(i)]
$\displaystyle \Phi(x,y)= x \prod_{p\le y} \left(1-\frac{1}{p}\right) + O\left( 2^{\pi(y)}\right)\quad (x\ge 1, y\ge 2)$
\item[]
\item[(ii)] $\displaystyle \Phi(x,y) = \frac{x \omega(u) -y}{\log y} + O\left(\frac{x}{(\log y)^2}\right) \quad (x\ge y\ge 2)$
\item[]
\item[(iii)] $\displaystyle \Phi(x,y) = \frac{x \omega(u)}{\log y} + O\left(\frac{y}{\log y}+\frac{x}{(\log y)^2}\right) \quad (x\ge 1,  y\ge 2)$
\item[]
\item[(iv)] $\displaystyle \Phi(x,y)-1\ll \frac{x}{\log y}\quad (x\ge 1, y\ge 2)$
\end{enumerate}
\end{lemma}
\begin{proof}
Part (i) is elementary (see e.g. de Bruijn \cite{Bru}). For (ii) see Tenenbaum \cite[Theorem III.6.3]{Ten}. Parts (iii) and (iv) follow easily from (ii).
\end{proof}

\begin{lemma}\label{omega}
We have
\begin{enumerate}[(i)]
\item  $\displaystyle |\omega'(u)|\le 1/\Gamma(u+1) \quad (u\ge 1)$
\item[]
\item $\displaystyle |\omega(u)-e^{-\gamma}| \ll 1/\Gamma(u+1) \quad (u\ge 1)$
\end{enumerate}
\end{lemma}
\begin{proof}
See Tenenbaum \cite[Theorems III.5.5, III.6.4]{Ten}. 
\end{proof}

In the proof of Theorem \ref{thmN} we will use the well-known fact (see for example \cite[Theorem I.5.5]{Ten}) 
$\limsup_{n\to \infty} \sigma(n)/(n \log \log n) = e^\gamma .$

\begin{proof}[Proof of Theorem \ref{thmN}]

(i) We use Lemma \ref{lem1}(ii). If $\sqrt{x}<y\le x$, then $M(x,y)=P(x)-P(y-0)$ because $\Phi(x,y)=1$ for $y\ge x\ge 1$.
Thus the result follows from \eqref{P} in this case. If $y\le \sqrt{x}$ we have
$$M(x,y) =P(x) -P(\sqrt{x})+ \sum_{y\le  n\le \sqrt{x}} \chi(n) \Phi\bigl(x/n, \sigma(n)+1\bigr).$$
We  approximate $\Phi$ by Lemma \ref{Phi}(iii). The contribution from the error term $O(x/(\log y)^2)$ is 
$$ \sum_{y\le  n\le \sqrt{x}} \chi(n)  \frac{x/n}{(\log n)^2} \ll \frac{x}{(\log y)^2},$$
and from the error term $O(y/\log y)$ it is
$$ \sum_{y\le  n\le \sqrt{x}} \chi(n)   \frac{\sigma (n)+1}{\log (\sigma(n)+1)} \ll 
\frac{\sqrt{x} \log\log x}{\log x} \sum_{y\le  n\le \sqrt{x}} \chi(n) \ll \frac{x \log \log 2x}{(\log x)^2},$$
which is acceptable. The contribution from the main term is
\begin{equation*}
x \sum_{y\le n\le \sqrt{x}} \frac{\chi(n)}{n \log(\sigma(n)+1)}\,\omega\left( \frac{\log x/n}{\log(\sigma(n)+1)} \right) .
\end{equation*} 
In the last sum, we replace the two occurrences of $\log(\sigma(n)+1)$ by
$\log n + O(\log\log\log(8n))$. Lemma \ref{omega} and \eqref{P} show that the resulting error is
$ \ll x (\log \log 2y)/(\log y)^2.$
We thus have
$$M(x,y) =P(x)+x \sum_{y\le  n\le \sqrt{x}} \frac{\chi(n)}{n \log n} \,\omega\left(\frac{\log x}{\log n} -1\right)   
+O\left(\frac{x \log \log 2y}{(\log y)^2}\right).$$
 Partial summation together with the estimates in Lemma \ref{omega} and \eqref{P} yields 
 \begin{equation*}
 M(x,y) =P(x)+x \int_y^{\sqrt{x}} \frac{c}{t (\log t)^2} \,\omega\left(\frac{\log x}{\log t} -1\right) \mathrm{d}t   
+O\left(\frac{x \log \log 2y}{(\log y)^2}\right).
\end{equation*}
The term with the integral simplifies to
$$\frac{cx}{\log x} \int_2^u \omega(s-1) \, \mathrm{d}s =\frac{cx}{\log x} \left(u \omega(u)-1\right) .$$
The result now follows from \eqref{P}.

(ii) Lemma \ref{lem1} shows that
\begin{equation*}
\begin{split}
0 \le N(x,y) - M(x,y) & = \sum_{\substack{n<y \\ \sigma(n)\ge y}} \chi(n) \Phi\bigl(x/n, \sigma(n)+1\bigr)\\
& \le  \sum_{\frac{y}{A\log\log 2y}< n<y } \chi(n) \Phi\bigl(x/n, \sigma(n)+1\bigr),\\
\end{split}
\end{equation*} 
for some suitable constant $A$. Splitting the range by powers of $2$ and using the estimate \eqref{P} and Lemma \ref{Phi} (iv), the last sum is 
$$
\ll   P(y)+\sum_{\frac{y}{A\log\log 2y}< n<y } \frac{x}{n (\log n)^2} \ll \frac{y}{\log y} +\frac{x \log \log 2 y}{(\log y)^2} .
$$
Hence (ii) follows from (i). 

(iii) From Lemmas \ref{lem1} and \ref{Phi} we have
\begin{equation*}
\begin{split}
[x]-M(x,y) &= \sum_{n<y} \chi(n)  \Phi\bigl(x/n, \sigma(n)+1\bigr) \\
& = \sum_{n<y} \chi(n)  \left( \frac{x}{n} \prod_{p\le \sigma(n)+1} \left(1-\frac{1}{p}\right)
+O\left(2^{\pi(\sigma(n)+1)}\right) \right) \\
& = x(1-\mu_y) + O \left(  \sum_{n<y} 2^{\pi(\sigma(n)+1)}\right)\\
& = x(1-\mu_y) + O \left(   2^{(1+o(1)) e^\gamma y \log\log y /\log y} \right),
\end{split}
\end{equation*}
since $\sigma(n) \le (1+o(1))e^\gamma n \log \log n$  and $\pi(y) \le (1+o(1))y/\log y$.

We omit the proof of (iv), since it is almost the same as that of (iii). 
\end{proof}

\begin{proof}[Proof of Corollary \ref{pro}]
(i) From Lemma \ref{lem1} and Lemma \ref{Phi} (iv) we have, with $\lambda=1/u$,
\begin{equation*}
\begin{split}
M_\lambda(x)-M(x,x^\lambda) & = \sum_{n< x^\lambda} \chi(n) \Phi\left(n^{1/\lambda -1}, \sigma(n)+1\right) \\
&= P(y) + O\left(\sum_{n\le y} \chi(n) \frac{n^{u -1}}{\log 2n}\right) \\
&= P(y) + O\left(\frac{x}{(\log y)^2}\right),
\end{split}
\end{equation*} 
by partial summation. The result now follows from Theorem \ref{thmN} and \eqref{P}.  
The proof of (ii) follows the same idea. In the end we need an estimate for
$$  \sum_{\sigma(n)< y} \chi(n) \frac{\sigma(n)^u}{n \log 2n}  .$$
We split this sum into two parts. The contribution from large $n$ is 
$$ \le \sum_{\frac{y}{A (\log y)^3} <n < y} \chi(n) \frac{y^u}{n\log 2n} \ll \sum_{\frac{y}{A (\log y)^3} <n < y}  \frac{x}{n(\log 2n)^2}
\ll \frac{x \log \log y}{(\log y)^2},$$
where $A$ is a positive constant such that $\sigma(n)\le \frac{y}{(\log y)^2}$ whenever $n\le \frac{y}{A (\log y)^3}$ and $y\ge 2$. The contribution
from small $n$ is 
$$  \le \sum_{n\le \frac{y}{A (\log y)^3} } \chi(n) \frac{(y(\log y)^{-2})^u}{n \log 2n}
\ll \frac{x}{(\log y)^2}\sum_{n\ge 1} \frac{1}{n(\log 2n)^2} \ll  \frac{x}{(\log y)^2}.$$
\end{proof}

\end{document}